\documentclass{article}
\usepackage[cp1251]{inputenc}
\usepackage[english]{babel}
\usepackage{amsfonts,amssymb,mathrsfs,amscd,amsmath,amsthm}
\usepackage{verbatim}
\usepackage{url}

\def\thtext#1{
  \catcode`@=11
  \gdef\@thmcountersep{. #1}
  \catcode`@=12
}

\def\threst{
  \catcode`@=11
  \gdef\@thmcountersep{.}
  \catcode`@=12
}

\theoremstyle{plain}
\newtheorem{thm}{Theorem}
\newtheorem{cor}[thm]{Corollary}

\newtheorem{prop}[thm]{Proposition}

\theoremstyle{definition}
\newtheorem{examp}{Example}
\newtheorem{rmk}{Remark}

 \catcode`@=11
 \def\.{.\spacefactor\@m}
 \catcode`@=12

\newcommand{\0}{\emptyset}
\renewcommand{\:}{\colon}
\renewcommand{\a}{\alpha}
\newcommand{\bX}{\overline{X}}
\newcommand{\bY}{\overline{Y}}
\newcommand{\cA}{\mathcal{A}}
\newcommand{\cC}{\mathcal{C}}
\newcommand{\cR}{\mathcal{R}}
\newcommand{\D}{\Delta}
\newcommand{\diam}{\operatorname{diam}}
\newcommand{\dis}{\operatorname{dis}}
\newcommand{\e}{\varepsilon}
\newcommand{\fc}{\mathfrak{c}}
\newcommand{\g}{\gamma}
\newcommand{\GH}{\operatorname{\mathcal{G\!H}}}
\newcommand{\hH}{\hat{H}}
\newcommand{\hQ}{\widehat{\mathbb{Q}}\strut}
\newcommand{\hR}{\widehat{\mathbb{R}}}
\newcommand{\hX}{\widehat{X}}
\newcommand{\hY}{\widehat{Y}}
\renewcommand{\l}{\lambda}
\newcommand{\N}{\mathbb{N}}
\newcommand{\om}{\omega}
\renewcommand{\r}{\rho}
\newcommand{\rom}[1]{\emph{#1}}

\renewcommand{\)}{\rom)}
\newcommand{\Q}{\mathbb{Q}}
\newcommand{\R}{\mathbb{R}}
\newcommand{\s}{\sigma}

\newcommand{\spe}{\supseteq}
\renewcommand{\ss}{\subset}
\newcommand{\sse}{\subseteq}
\newcommand{\St}{\operatorname{St}}
\newcommand{\tT}{\widetilde{T}}
\newcommand{\tX}{\widetilde{X}}
\newcommand{\toGH}{\xrightarrow{\GH}}
\newcommand{\x}{\times}
\newcommand{\VGH}{\operatorname{\mathcal{VG\!H}}}

\begin{document}
\title{Clouds in Gromov--Hausdorff Class: their completeness and centers}
\author{S.~I.~Bogataya, S.~A.~Bogatyy, V.~V.~Redkozubov, A.~A.~Tuzhilin}
\date{}
\maketitle

\begin{abstract}
We consider the proper class of all metric spaces endowed with the Gromov--Hausdorff distance. Its maximal subclasses, consisting of the spaces on finite distance from each other, we call clouds. Multiplying all distances in a metric space by the same positive real number, we obtain a similarity transformation of the Gromov--Hausdorff class. In our previous work, we observed that with such a transformation, some clouds can jump to others. To characterize the phenomenon, we studied the stabilizers of the similarity action. In this paper, we prove that every cloud with a nontrivial stabilizer has a center, i.e., a metric space for which all similarities from the stabilizer generate a new space at zero distance. Moreover, the center is unique modulo zero distance. The proof is based on the cloud completeness theorem.\footnote{The authors thank P.A.~Borodin, E.A.~Reznichenko and O.V.~Sipacheva. The work was done in Moscow State University, A.A~Tuzhilin was supported by RSF grant N21-11-00355.}

\end{abstract}

%%%%%%%%%%%%%%%%%%%%%%%%%%%%%%
\section{Introduction}
%%%%%%%%%%%%%%%%%%%%%%%%%%%%%%
The present work is devoted to the geometry of the Gromov--Hausdorff distance~\cite{Edwards, Gromov1981, Gromov1999, BurBurIva} defined on the collection of all non-empty metric spaces. It is well-known that this distance is a generalized pseudometric vanishing on isometric spaces (``generalize'' means that infinite distances may occur, and the prefix ``pseudo-'' that the distances between different spaces may vanish). Traditionally, the distance is studied on the Gromov--Hausdorff space, in which all metric spaces are compact~\cite{BurBurIva} and considered up to isometry.

M.~Gromov in his ``Metric structures for Riemannian and non-Riemannian spaces''~\cite{Gromov1999} made a short remark: ``One can also make a moduli space of isometry classes of non-compact spaces $X$ lying within a finite Hausdorff distance from a given $X_0$, e.g. $X_0=\R^n$. Such moduli spaces are also complete and contractible.'' This observation was not proved in~\cite{Gromov1999} because it probably seemed obvious.

In the present paper we will prove the completeness of the Gromov--Hausdorff distance in general. Let us note that our proof is not difficult indeed, and we present it mainly to show how the use of correspondences instead of mappings may simplify the construction of direct limits. Our study of stabilizers is based on the completeness theorem, which we considered possible to use only after proof has been presented. Moreover, we cannot prove the second statement of the last sentence of Gromov's quotation and we fully admit that it may turn out to be non-trivial or even incorrect.

In von Neumann-Bernays-G\"odel (NGB) set theory, all objects are called \emph{classes}. A distance function can be specified on each class. We will use the following terminology:
\begin{itemize}
\item a \emph{distance function\/} or, in short, a \emph{distance\/} on a class $\cA$ is an arbitrary mapping $\r\:\cA\x\cA\to[0,\infty]$, for which always $\r(x,x)=0$ and $\r(x,y)=\r(y,x)$ (symmetry) hold;
\item if the distance satisfies the triangle inequality, i.e., if $\r(x,z)\le\r(x,y)+\r(y,z)$ is always satisfied, then $\r$ is called \emph{generalized pseudometric\/} (the word ``generalized'' corresponds to the possibility of taking the value $\infty$);
\item if the positive definiteness condition is additionally satisfied for a generalized pseudometric, i.e., if $\r(x,y)=0$ always implies $x=y$, then we call such $\r$ \emph{generalized metric\/};
\item finally, if $\r$ does not take the value $\infty$, then in the above definitions we will omit the word ``generalized'', and sometimes, to emphasize the absence of $\infty$, we call this distance \emph{finite}.
\end{itemize}

As is customary in metric geometry, instead of $\r(x,y)$ we write $|xy|$ as a rule.

Let us now recall the definitions we need from metric geometry, namely, the Hausdorff and Gromov--Hausdorff distances.

%%%%%%%%%%%%%%%%%%%%%%%%%%%%%%
\subsection{Hausdorff distance}
%%%%%%%%%%%%%%%%%%%%%%%%%%%%%%
Let $X$ be an arbitrary metric space, $x\in X$, $r>0$ and $s\ge 0$ be real numbers. By $U_r(x)$ and $B_s(x)$ we denote, respectively, the \emph{open\/} and \emph{closed balls\/} centered at the point $x$ and with the radii $r$ and $s$. If $A$ and $B$ are non-empty subsets of $X$, then we put $|AB|=|BA|=\inf\bigl\{|ab|:a\in A,\,b\in B\bigr\}$. Next, we define the \emph{open $r$-neighborhood of the set $A$} by setting $U_r(A)=\bigl\{x\in X:|xA|<r\bigr\}$. Finally, the \emph{Hausdorff distance\/} between $A$ and $B$ is the value
$$
d_H(A,B)=\inf\bigl\{r:A\ss U_r(B),\ \ B\subset U_r(A)\bigr\}.
$$
The Hausdorff distance is a generalized pseudometric: it can be infinite, as in the case of the straight line $\R$ and any of its points, and also equal to zero between different subsets, for example, between the segment $[0,1]$ and the interval $(0,1)$. It is important to note that the distance between a set and its dense subset is zero. Therefore, the distance between a metric space and its completion is zero. As a result, one can often assume the completeness of the space under consideration. Nevertheless, on the set consisting of all non-empty bounded closed subsets of a metric space $X$, the Hausdorff distance is a metric.

%%%%%%%%%%%%%%%%%%%%%%%%%%%%%%
\subsection{Gromov--Hausdorff distance}
%%%%%%%%%%%%%%%%%%%%%%%%%%%%%%

We denote by $\VGH$ the class consisting of all non-empty metric spaces. On this class, we define a distance function called the \emph{Gromov--Hausdorff distance\/}:
$$
d_{GH}(X,Y)=\inf\bigl\{d_H(X',Y'):X',Y'\ss Z\in\VGH,\,X'\approx X,\,Y'\approx Y\bigr\},
$$
where for the metric spaces $U$ and $V$ the expression $U\approx V$ means that these spaces are isometric.

The following theorem is well known.

\begin{thm}[\cite{BurBurIva}]
The Gromov--Hausdorff distance is a generalized pseudometric vanishing on each pair of isometric spaces.
\end{thm}

This theorem allows us to investigate the Gromov--Hausdorff distance on a ``less wild'' so called \emph{Gromov--Hausdorff class} $\GH$ obtained from $\VGH$ by factorization over the isometry-equivalence for which two metric spaces are equivalent iff they are isometric.

Together with $\GH$, we consider another class $\GH_0$ obtained from $\VGH$ by factorization over the zero-value equivalence for which $X$ is equivalent $Y$ iff $d_{GH}(X,Y)=0$. Clearly that the Gromov--Hausdorff distance is a generalized metric on $\GH_0$.

Consider the relation $\sim_1$ on $\VGH$: we say that $X,Y\in\VGH$ are in relation $\sim_1$ if and only if $d_{GH}(X,Y)<\infty$. It is easy to see that $\sim_1$ is an equivalence, and it generates the corresponding equivalences on both $\GH$ and $\GH_0$. The equivalence classes of this relation will be called {\it clouds}. It is clear that the Gromov--Hausdorff distance between elements of the same cloud is finite, and between elements of different clouds is infinite. The class of all clouds we denote by $\GH_c$. We have evident canonical projections $\VGH\to\GH\to\GH_0$ preserving the Gromov--Hausdorff distance and projection $\GH_0\to\GH_c$. To simplify the statements of the theorems, it is convenient to also consider the equivalence class of $\sim_1$ in $\GH_0$ as a cloud, i.e., under a cloud it is often convenient to consider its projection to the class $\GH_0$.

Since our base object is a metric space, we will denote by Latin capital letter $X$, $Y$, etc., a metric space with a fixed metric, i.e. $X$ is an element in $\VGH$. But since we do not distinguish isometric spaces at all, we retain the same notation for metric spaces, considered up to isometry, i.e. $X$ is also understood as an element in $\GH$. By $X_0=(X)_0$ we denote the class of all spaces $Y$ such that $d_{GH}(X,Y)=0$, i.e. $X_0$ is class of all spaces on zero distance from $X$ and can be considered as element of $\GH_0$. The cloud of the space $X$ we denote by $[X]$ and clouds can be considered as elements of $\GH_c$. It is convenient to call the elements of $\GH$, $\GH_0$, and $\GH_c$ by \emph{metric spaces considered up to the corresponding equivalence}.

Let us describe a few simplest properties of the four classes $\VGH$, $\GH$, $\GH_0$, and $\GH_c$. By $\D_1$ we denote a one-point metric space. Also, for a metric space $(X,\varrho )$ and a real number $\l>0$, we denote by $\l X=(X,\l \varrho )$ the metric space that is obtained from $X$ by multiplying its all distances by $\l$. The transformation $H_\l\:\VGH\to\VGH$, $H_\l\:X\mapsto\l X$ for $\l>0$ we call \emph{similarity with the coefficient $\l$}.

\begin{thm}[\cite{BurBurIva}]\label{thm:estim}
For any metric spaces $X$ and $Y$,
\begin{enumerate}
\item\label{thm:estim:1} $2d_{GH}(\D_1,X)=\diam X$\rom;
\item\label{thm:estim:2} $2d_{GH}(X,Y)\le\max\{\diam X,\diam Y\}$\rom;
\item\label{thm:estim:3} if the diameter of $X$ or $Y$ is finite, then $\bigl|\diam X-\diam Y\big|\le2d_{GH}(X,Y)$.
\item\label{thm:estim:4} if the diameter of $X$ is finite, then for any $\l>0$ and $\mu>0$ we have $d_{GH}(\l X,\mu X)=\frac12|\l-\mu|\diam X$, whence it immediately follows that the curve $\g(t):=t\,X$ is shortest between any of its points, and the length of such a segment of the curve is equal to the distance between its ends\rom;
\item\label{thm:estim:5} for any $\l>0$, we have $d_{GH}(\l X,\l Y)=\l\,d_{GH}(X,Y)$.
\end{enumerate}
\end{thm}

Property~(\ref{thm:estim:5}) implies that similarities are well defined on the classes $\GH$, $\GH_0$, and $\GH_c$: $\l X_0=(\l X)_0$ and $\l[X]=[\l X]$. We keep the same notations $H_\l$ for the corresponding mappings defined on $\GH$, $\GH_0$, and $\GH_c$.

Returning to the contractibility of the cloud, we note that formulas~(\ref{thm:estim:4}) and~(\ref{thm:estim:5}) illustrate the existence of a canonical contraction of a cloud of bounded metric spaces to the one-point space $\D_1$. Formula~(\ref{thm:estim:5}) means that the similarity is continuous in space, but formula~(\ref{thm:estim:4}) in all other clouds does not guarantee continuity with respect to the contraction parameter $\l$.

The authors constructed~\cite[Corollary~5.9]{BT21},~\cite{BBRT22V} an example of a space $X$ such that the spaces $X$ and $\l X$ lie in the same cloud if and only if $\l=1$. This means that, in the general case, the similarity cannot contract the cloud by itself. Therefore, the authors believe that the statement about the contractibility of any cloud is currently a hypothesis.

The main question investigated in the paper can be informally formulated as follows.

\emph{Is it possible to find a center in a cloud of unbounded space, i.e. such a space that, similarly to the $\D_1$, remains fixed under the action of all $H_\l$ and attracts all other spaces as $\l\to0$}?

To do this, we first study the \emph{stabilizer} --- the set of multipliers $\l$ that do not remove space out of its cloud.

Next, we show that a cloud with a non-trivial stabilizer has a common fixed point (in $\GH_0$) for all similarities from the stabilizer, and that the fixed point is unique (Corollary~\ref{cor:St}). This point is called \emph{the center of the cloud}.

Next, we study the properties of spaces from the center of the cloud at the level of the class $\GH$.

In~\cite{BT21,BT22P,BB22T} we investigated the stabilizers of geometric progressions. In the present paper we describe in details so-called ``discrete hedgehog with intrinsic metric''. Such spaces consist of subsets of positive real numbers with added $0$ and endowed with the distance equal to the sum of coordinates: for $x_1$ and $x_2$, the distance between them is $x_1+x_2$. In particular, we show that the family of such hedgehogs is closed in the following sense: each complete metric space on zero distance from a hedgehog is isometric to a hedgehog (of a set with multiplicities of points) (Theorem~\ref{thm:ezh0disthatX}). Let us mention that the zero distance between even boundedly compact metric spaces does not imply that the spaces are isometric~\cite{Ghanaat}. Two locally compact countable bounded metric spaces at distance zero can even be non-homeomorphic (Example~\ref{examp:Tuzhilin}). We put
\begin{gather*}
\St X=\{\l\in\R_+:\text{$\l X$ is isometric to $X$}\}, \\
\St_0X=\bigl\{\l\in\R_+:d_{GH}(X,\l X)=0\bigr\}=\bigl\{\l\in\R_+:H_\l(X_0)=X_0\bigr\},\\
\St[X]=\bigl\{\l\in\R_+:d_{GH}(X,\l X)<\infty\bigr\}=\Bigl\{\l\in\R_+:H_\l\bigl([X]\bigr)=[X]\Bigr\}.
\end{gather*}
It is clear that there are inclusions
$$
\St X\sse\St_0X\sse\St[X]\sse \R_+.
$$
It is also important to note that
\begin{enumerate}
\item\label{enum:Stab:1} if some spaces $X$ and $Y$ are isometric, then $\St X=\St Y$;
\item\label{enum:Stab:2} if $d_{GH}(X,Y)=0$, then $\St_0X=\St_0Y$;
\item\label{enum:Stab:3} if $d_{GH}(X,Y)<\infty$, then $\St[X]=\St[Y]$.
\end{enumerate}

At last, we investigated in~\cite{BT22P} the similarity transformations action on a normed vector space $V$. The corresponding stabilizer of a set $X\ss V$, i.e. of an element $X\in 2^V$ will be denoted by $\St_VX$. Thus,
$$
\St_VX=\{\l\in\R_+:\l X=X\}.
$$
In this case, there is a longer chain of inclusions
$$
\St_VX\sse\St X\sse\St_0X\sse\St[X]\sse\R_+.
$$
P.A.~Borodin drew our attention to the fact that it is natural to consider the hedgehog $\hX$ over the set $X$ as a subset of the Banach space $\ell_1(X)$. Hedgehogs give examples with different ratios of stabilizers. For example, $\St_{\ell_1}\hX=\{1\}$, $\St\hX=\St_\R X$ for every $X\ne\0$ (Corollary~\ref{cor:ezhdill}). In particular, $\St_\R H=\St\hH=H$ for every subgroup $H\sse\R_+$.

For a cloud $[X]$ with a nontrivial stabilizer $\St[X]\ne\{1\}$ by \emph{center\/} we call
$$
Z[X]=\bigl\{Y\in[X]:\text{$Y$ is complete and $\St_0Y=\St[X]$}\bigr\}.
$$
If $\St[X]\ne\{1\}$ then $Z[X]\ne\0$ and $d_{GH}(Y_1,Y_2)=0$ for any $Y_1,Y_2\in Z[X]$ (Corollary~\ref{cor:St}). And vice versa, if $d_{GH}(Y_1,Y_2)=0$ and $Y_1\in Z[X]$, then $Y_2\in Z[X]$, see~(\ref{enum:Stab:2}). According to Proposition~\ref{prop:St} the center of the cloud can be described as the family of all spaces with a non-trivial stabilizer; and the stabilizer of any such element is maximally possible for spaces from the cloud: it is equal to $\St[X]$.

The hedgehog $\hQ_+$ over positive rational numbers is interesting because that the complete space $Y$ has zero distance to this hedgehog $\hQ_+$ if and only if it is isometric to a hedgehog over some countable dense subset $\R_+$ (Corollary~\ref{cor:ezh0distQ}) and for it $|\St Y|\le\aleph_0$, $\St_0Y=\R_+$ (Corollary~\ref{cor:isometryhatQ}). There is exactly $\fc=2^{\aleph_0}$ of such pairwise non-isometric spaces (Theorem~\ref{thm:StabilhatQ}).

For any countable subgroup $H\sse\R_+$ ($|H|\le\aleph_0$) there is a hedgehog $\hX$, such that $d_{GH}(\hX,\hQ_+)=0$ (Corollary~\ref{thm:StabilhatQ}) and $\St_\R X=\St\hX=H\subset \R_+=\St_0\hX$ (Corollary~\ref{cor:isometryH}).

The hedgehog $\hR_+$ over positive numbers is interesting because that at zero distance from it there are exactly $2^\fc=2^{2^{\aleph_0}}$ pairwise non-isometric complete spaces (they all have weight $\fc$) (Theorem~\ref{thm:StabilhatR}). For any subgroup $H\sse\R_+$ with $|H|=\fc$ (there are exactly $2^\fc$ of such subgroups), $d_{GH}(\hH,\hR_+)=0$ and the chain of inclusions
$$
\St_{\ell_1(\fc)}\hH=\{1\}\ss H=\St_\R H= \St\hH\ss\R_+=\St_0\hH
$$
is valid.

%%%%%%%%%%%%%%%%%%%%%%%%%%%%%%
\section{Completeness of the Gromov--Hausdorff class}
%%%%%%%%%%%%%%%%%%%%%%%%%%%%%%

The purpose of this subsection is to prove one more fundamental property of the Gromov--Hausdorff class. Since we will use in our proof the technique of correspondences, let us recall the definition and some necessary well-known results.

Let $X$ and $Y$ be some sets. A \emph{relation\/} between $X$ and $Y$ is an arbitrary subset $R\ss X\x Y$. If $X,Y\in\VGH$ are non-empty metric spaces, and $\s$ is a non-empty relation, its \emph{distortion $\dis\s$} is defined as follows:
$$
\dis\s=\sup\Bigl\{\bigl||xx'|-|yy'|\bigr|:(x,y),\,(x',y')\in\s\Bigr\}.
$$

A relation $R\ss X\x Y$ between sets $X$ and $Y$ is called a \emph{correspondence\/} if for any $x\in X$ there exists $y\in Y$, and for any $y\in Y$ there exists $x\in X$, such that $(x,y)\in R$. Thus, the correspondences can be considered as multivalued surjective mappings.

For a correspondence $R\ss X\x Y$ and $x\in X$, we put $R(x)=\big\{y\in Y:(x,y)\in R\big\}$ and call $R(x)$ the \emph{image of the element $x$ under the correspondence $R$}. By $\cR(X,Y)$ we denote the set of all correspondences between $X$ and $Y$. The following result is well-known.

\begin{thm}[see~\cite{BurBurIva}]\label{thm:dis}
For any $X,Y\in\VGH$, it holds
$$
d_{GH}(X,Y)=\frac12\inf\bigl\{\dis R:R\in\cR(X,Y)\bigr\}.
$$
\end{thm}

\begin{thm}\label{thm:GHcomplete}
The Gromov--Hausdorff class $\VGH$ is complete. In particular, all clouds are complete.
\end{thm}

\begin{proof}
Let $X_1,X_2,\ldots $ be a fundamental sequence. Without loss of generality, we will assume that $2d_{GH}(X_n,X_{n+1})\le1/2^n$. For each $n$, choose a correspondence $R_n\in\cR(X_n,X_{n+1})$ such that $\operatorname{dis}R_n<1/2^n$. We put $\bX=\sqcup_{n=1}^\infty X_n$ and for each $x_1\in X_1$ we call the sequence $x_1,x_2,\ldots \in \bX$ a \emph{thread starting at $x_1$}, if for any $n\ge1$ it holds $x_{n+1}\in R_n(x_n)$. We denote the set of all such threads by $N(\bX)$. For threads $\nu=(x_1,x_2,\ldots)$ and $\nu'=(x'_1,x'_2,\ldots)$ consider the sequence $|x_1x'_1|,|x_2x'_2|,\ldots$ of real numbers. Since $\dis R_n<1/2^n$, then $\bigl||x_nx'_n|-|x_{n+1}x'_{n+1}|\bigr|<1/2^n$, whence for any $m\ge1$ we have
\begin{multline*}
\bigl||x_nx'_n|-|x_{n+m}x'_{n+m}|\bigr|=\\
\bigl||x_nx'_n|-|x_{n+1}x'_{n+1}|+|x_{n+1}x'_{n+1}|-\cdots-|x_{n+m}x'_{n+m}|\bigr|\le\\
\sum_{k=1}^m\bigl||x_{n+k-1}x'_{n+k-1}|-|x_{n+k}x'_{n+k}|\bigr|<\sum_{k=1}^m\frac1{2^{n+k-1}}<\frac1{2^{n-1}},
\end{multline*}
therefore, the sequence of numbers $|x_1x'_1|,|x_2x'_2|,\ldots$ is fundamental and, therefore, there is a limit, which we denote by $|\nu\nu'|$. It is clear that these limits define a distance function on $N(\bX)$. Since for any three threads $\nu=(x_1,x_2,\ldots)$, $\nu'=(x'_1,x'_2,\ldots)$, $\nu''=(x''_ 1,x''_2,\ldots)$, and for each $n$, the triangle inequalities hold for the points $x_n,x'_n,x''_n$, the same is true for the distances between $\nu$, $\nu'$ and $\nu''$. Thus, we have defined a pseudometrics on $N(\bX)$. Having factorized the resulting pseudometric space with respect to the equivalence relation generated by zero distances, we obtain a metric space, which we denote by $X$. For a thread $\nu \in N(\bX)$, the class of this equivalence containing $\nu $ will be denoted by $[\nu]$.

For each $n$, consider the relation $R'_n\ss X\x X_n$ defined as follows: for each $x\in X$, consider all threads $\nu=(x_1,x_2,\ldots,x_n,\ldots)\in x$ and put all $(x,x_n)$ into $R'_n$. Since each $x_n$ belongs to some thread, then $R'\in\cR(X,X_n)$. Let us estimate the distortion of the correspondence $R '$. Choose arbitrary $x,x'\in X$, threads $\nu=(x_1,x_2,\ldots,x_n,\ldots)\in x$ and $\nu'=(x'_1,x'_2,\ldots,x'_n,\ldots)\in x'$, then $|\nu\nu'|=\lim_{k\to\infty}|x_kx'_k|$, and as shown above, $\bigl||x_nx'_n|-|x_{n+m}x'_{n+m}|\bigr|<1/2^{n-1}$, whence $\bigl||x_nx'_n|-|\nu\nu'|\bigr|\le1/2^{n-1}$. Thus, $\dis R'\le1/2^{n-1}$, so $X_n\toGH X$.
\end{proof}

\begin{rmk}
As we mentioned in Introduction, the distance between the corresponding elements of the ``quotients'' $\GH$ and $\GH_0$ are the same, that is why Theorem~\ref{thm:GHcomplete} implies completeness of all the classes and their clouds.
\end{rmk}

\begin{rmk}\label{rmk:complete}
Without loss of generality, we can always assume that the limit space of the fundamental sequence is complete.
\end{rmk}

%%%%%%%%%%%%%%%%%%%%%%%%%%%%%%
\section{Collective isometric embedding}
%%%%%%%%%%%%%%%%%%%%%%%%%%%%%%
One of the famous M.~Gromov theorems states that for each totally bounded family $\cC$ consisting of non-empty compact metric spaces, there exists a compact set $K$ of the Banach space $\ell_\infty$ such that each space $X\in\cC$ can be isometrically embedded into $K$, see~\cite{BurBurIva} for example. We need a rather simple result: let we are given with
\begin{itemize}
\item a family $\cC$ of (not necessarily compact) metric spaces;
\item a tree $G$ (connected acyclic graph) with vertex set $\cC$ and edge set $E$;
\item a ``weight'' function $\om$: for each $e\in E$ we choose an arbitrary correspondence $R_{XY}\in\cR(X,Y)$, we  put $R_{YX}=R_{XY}^{-1}\in\cR(Y,X)$, and we define $\om(e)=\frac12\dis R_{XY}=\frac12\dis R_{YX}$;
\item we assume that for all $e\in E$ we have $0<\om(e)<\infty$.
\end{itemize}
We wish to construct a metric on $Z:=\sqcup_{X\in\cC}X$ which extends the metrics of all the $X$ in such a way that $d_H(X,Y)=\om(XY)$ for all $XY\in E$. To do that, we use the ideas described in~\cite{TuzLect}.

Given metric spaces $X$ and $Y$, consider an arbitrary correspondence $R\in\cR(X,Y)$. Suppose that $0<\dis R<\infty$. Extend the metrics of $X$ and $Y$ upto a symmetric function $|\cdot|_R$ defined on pairs of points from $X\sqcup Y$: for $x\in X$ and $y\in Y$ put
\begin{equation}\label{eq:DistForDistort}
|xy|_R=|yx|_R=\inf\Bigl\{|xx'|+|yy'|+\frac12\dis R:(x',y')\in R\Bigr\}. \\
\end{equation}

\begin{prop}[\cite{TuzLect}]\label{prop:correspondence-to-pseudometric}
Under above assumptions, $|\cdot|_R$ is a metric on $X\sqcup Y$, and $d_H(X,Y)=\frac12\dis R$, where $d_H$ is the corresponding Hausdorff distance.
\end{prop}

\begin{thm}\label{thm:CollectiveIsoEmbed}
For $(\cC,G,\om)$, there exists an extension of metrics from all $X\in\cC$ to the whole $Z=\sqcup_{X\in\cC}X$, which satisfies the conditions $d_H(X,Y)=\om(XY)$ for all edges $XY$ of the tree $G$.
\end{thm}

\begin{proof}
For each edge $XY\in E$, we use the formula~(\ref{eq:DistForDistort}) to construct the distance on $X\sqcup Y$. After that we extend the distances in the way to generate a gluing metric. Namely, let us choose arbitrary $X,Y\in\cC$ and any points $x\in X$, $y\in Y$. Consider the unique path $X_0=X,X_1,\cdots,X_n=Y$ in the graph $G$ joining $X$ and $Y$, and put
$$
|xy|=\inf\Bigl\{\sum_{i=1}^n|x_{i-1}x_i|:x_0=x,\,x_n=y,\,x_k\in X_k,\,k=1,\ldots,n-1\Bigr\}.
$$
The verification that we have got a metric on $Z$ is direct. Evidently, the metric preserves the distances between point in all $X\sqcup Y$, $XY\in E$, thus, Proposition~\ref{prop:correspondence-to-pseudometric} implies that for such $X$, $Y$ we have $d_H(X,Y)=\frac12\dis R=\om(e)$. The proof is complete.
\end{proof}

\begin{cor}\label{cor:CollectiveIsoEmbed}
Let a metric space $T$, a sequences of metric spaces $T_n$, and positive numbers $M_n>0$, $n\in\N$, be such that $d_{GH}(T,T_n)<M_n$. Suppose also that $T$ is not isometric to $T_n$ for all $n\in\N$. Then there is a metric space $Z$ containing isometric copies all the spaces $T$, $T_n$, $n\in\N$, in such a way that $d_H(T,T_n)<M_n$.
\end{cor}

\begin{proof}
Since $d_{GH}\bigl(T,T_n\bigr)<M_n$ and $T$ is not isometric to $T_n$, then for any $n\in\N$ we can take a correspondence $R_n\in\cR(T,T_n)$ with $0<\dis R_n<2M_n$. Let $G$ be the graph with the vertex set $\cC:=\{T_n\}_{n\in\N}\sqcup\{T\}$ and the edge set $E:=\{TT_n\}_{n\in\N}$. Define the weight function $\om\:E\to\R$ by $\om(TT_n)=\frac12\dis R_n$. Clearly that $G$ is a tree, and we can apply Theorem~\ref{thm:CollectiveIsoEmbed} to the triple $(\cC,G,\om)$ to obtain a metric space $Z$ isometrically containing $T$ and the sequence $\{T_n\}_{n=1}^\infty$ in a way that $d_H(T,T_n)=\om(TT_n)=\frac12\dis R_n<M_n$.
\end{proof}

%%%%%%%%%%%%%%%%%%%%%%%%%%%%%%
\section{Cloud center}
%%%%%%%%%%%%%%%%%%%%%%%%%%%%%%

From a metric space $X$, we define the function $d\:(0,\infty)\to[0,\infty]$ as follows: $d(\l)=d_{GH}(X,\l X)$.

\begin{prop}\label{prop:dist}
The function $d\:(0,\infty )\to[0,\infty]$ satisfies the conditions\/\rom:
\begin{enumerate}
\item\label{enum:FuncD:1} $d(1)=0$\rom;
\item\label{enum:FuncD:2} $d(\l^{-1})=\l^{-1}d(\l)$\rom;
\item\label{enum:FuncD:3} $d(\l\mu)\le d(\l)+\l d(\mu)$.
\end{enumerate}
\end{prop}

\begin{proof}
(\ref{enum:FuncD:1}) Is evident.

(\ref{enum:FuncD:2}) Indeed,
$\l^{-1}d(\l)=\l^{-1}d_{GH}(X,\l X)=d_{GH}(\l^{-1}X,\l^{-1}\l X)=d(\l^{-1})$.

(\ref{enum:FuncD:3})  According to the triangle inequality, the inequality
\begin{multline*}
d(\l\mu)=d_{GH}(X,\l\mu X)\le d_{GH}(X,\l X)+d_{GH}(\l X,\l\mu X)=\\
d(\l)+\l\,d_{GH}(X,\mu X)=d(\l)+\l d(\mu)
\end{multline*}
holds.
\end{proof}

\begin{prop}\label{prop:distlambda}
If $\l<1$ and $d=d(\l)<\infty$, then
$$
d(\l^n)\le\frac{1-\l^n}{1-\l}\,d<\frac{1}{1-\l}\,d<\infty.
$$
\end{prop}

\begin{proof}
According to Proposition~\ref{prop:dist}, by induction we have the inequality
$$
d(\l^n)\le d(\l)+\l\,d(\l^{n-1})\le d+\l\frac{1-\l^{n-1}}{1-\l}\,d=\frac{1-\l^n}{1-\l}\,d<\frac{1}{1-\l}\,d<\infty.
$$
\end{proof}

\begin{prop}
If $\l\in\St[X]$, then there is a complete space $X_\l\in [X]$ such that $d_{GH}(X_\l,\l X_\l)=0$.
\end{prop}

\begin{proof}
Without loss of generality, we can assume that $\l<1$. Let $\r$ be the metric of the space $X$. Let us prove that the sequence of metric spaces $X_n=\l^nX\in[X]$ is a fundamental sequence in the cloud $[X]$. Let $d=d(\l)=d_{GH}(X_1,X)<\infty$. According to Proposition~\ref{prop:distlambda}, for $m>n>1$ the inequality
$$
d_{GH}(X_m,X_n)=\l^nd_{GH}(X_{m-n},X)<\l^n\frac{1}{1-\l}\,d
$$
holds. According to the completeness Theorem~\ref{thm:GHcomplete}, this sequence has a limit $(X_\l,\r_\l)=\lim_{n\to\infty}X_n\in[X]$. Then
$$
d_{GH}\bigl((X_\l,\l\r_\l),(X_\l,\r_\l)\bigr)=\lim_{n\to\infty}d_{GH}\bigl(X_{n+1},X_n\bigr)=0.
$$
\end{proof}

\begin{prop}\label{prop:St}
If $\l,\mu\in\St[X]$, then $d_{GH}(X_\l,X_\mu)=0$.
\end{prop}

\begin{proof}
Without loss of generality, we can assume that $\l,\mu<1$. Let us consider the space $X_{\l\mu}\in[X]$. According to Proposition~\ref{prop:distlambda}, for $n>1$ the inequality
$$
d_{GH}\bigl((X,(\l\mu)^{n}\r),(X,\mu^{n}\r)\bigr)=
\mu^nd_{GH}\bigl((X,\l^{n}\r),(X,\r)\bigr)<\mu^n\frac{1}{1-\l}\,d
$$
holds. Hence $d_{GH}(X_{\l\mu},X_\mu)=0$ and $d_{GH}(X_{\l\mu },X_\l)=0$, respectively.
\end{proof}

\begin{cor}\label{cor:St}
For any cloud $[X]$ with $\St[X]\ne \{1\}$, there exists a complete metric space $X_0\in Z[X]$. For any two spaces $X_0,Y_0\in Z[X]$, the equality $d_{GH}(X_0,Y_0)=0$ holds.
\end{cor}

\begin{proof}
According to Proposition~\ref{prop:St}, the space $X_0=X_\l$, $\l\in\St[X]\setminus\{1\}$, is well defined, i.e., it does not depend on the choice of such a stabilizer element. Therefore $\St_0X_0\spe\St[X]$. So we have the equality $\St_0X_0=\St[X]$.

Now, let $X_0$ and $Y_0$ be two such spaces. Let $\l\in\St[X]$ and $\l\ne1$. Then
$$
d_{GH}(X_0,Y_0)=d_{GH}\bigl(\l X_0,\l Y_0\bigr)=\l\,d_{GH}(X_0,Y_0).
$$
Therefore $d_{GH}(X_0,Y_0)=0$.
\end{proof}

%%%%%%%%%%%%%%%%%%%%%%%%%%%%%%
\section{Discrete hedgehog}
%%%%%%%%%%%%%%%%%%%%%%%%%%%%%%

Let $X\sse\R_+$ be a non-empty subset. Let each point $x\in X$ be assigned a multiplicity, some nonzero cardinal $\tau_x$. Let $\cA$ be a set such that $\tau_x\le |\cA|$ for every point $x\in X$. Here and throughout the paper, the cardinality of a set $S$ is denoted by the symbol $|S|$. A set $X$ with multiplicities of points can be represented as a subset $A\ss X\x\cA$ such that the section $A_x=\{(x,\a)\in A\}$ has cardinality $\tau_x$. For the uniqueness of such a correspondence, it is convenient to assume that $\cA$ is the set of ordinals, and any section $A_x$ is the beginning of the ordinals of the required cardinality $\tau_x$. For the convenience of the geometric representation, we will identify $A$ with the set $X$, on which it is assumed that a structure of multiplicities of points is fixed. Unless otherwise stated, we always mean a hedgehog over a set each point of which has multiplicity $1$.

A discrete hedgehog is a set $\hX=\{0\}\cup A$ with intrinsic metric:
$$
\r\bigl((x_1,\a_1),(x_2,\a_2)\bigr)=x_1+x_2\ \ \text{for $(x_1,\a_1)\ne(x_2,\a_2)$}.
$$
The point $0$ of hedgehog $\hX$ we denote by $0_X$.

Recall that the \emph{density $d(X)$} of a topological space $X$ is the minimum cardinality of dense subsets.

Verification of the following assertions will not cause difficulties, so we give them without proof.

\begin{prop}\label{prop:propert}
For a non-empty set $X\sse\R_+$, the discrete hedgehog $\hX$ has the following properties\rom:
\begin{enumerate}
\item\label{enum:propert:1} The hedgehog $\hX$ is a complete metric space.
\item\label{enum:propert:2} For any number $\l\in\R_+$ the space $\l\hX$ is naturally identified with hedgehog of the set $\l X$\rom: $\l\hX=\widehat{\l X}$.
\item\label{enum:propert:3} The hedgehog $\hX$ is bounded if and only if the set $X$ is bounded.
\item\label{enum:propert:4} All points of the subset $X\ss\hX$ are discrete.
\item\label{enum:propert:5} The equality $d(\hX)=|A|+1$ is true.
\end{enumerate}
\end{prop}

\begin{prop}[\cite{Borisova}]\label{prop:density}
For metric spaces $X$ and $Y$ with $d_{GH}\bigl(X,Y)=0$ the equality $d(X)=d(Y)$ is true.
\end{prop}

\begin{prop}\label{prop:ezhdist}
For two subsets $X,Y\sse\R_+$ the following properties are equivalent\/\rom:
\begin{enumerate}
\item\label{enum:ezhdist:1} Hedgehogs $\hX$ and $\hY$ are isometric.
\item\label{enum:ezhdist:2} The sets $X$ and $Y$ are the same: $X=Y$ (as sets with multiplicities of points).
\end{enumerate}
\end{prop}

\begin{proof}
The implication $(\ref{enum:ezhdist:2})\Rightarrow(\ref{enum:ezhdist:1})$ is obvious.

$(\ref{enum:ezhdist:1})\Rightarrow(\ref{enum:ezhdist:2})$.
Without loss of generality, we can assume that $|X|\geq 2$. Let $h\:\hX\to\hY$ be an isometry. Since the points $0_X$ and $0_Y$ are the only metrically intermediate points, then $h(0_X)=0_Y$. Hence $h(X)=Y$. Since the mapping $h$ is isometric it follows
$$
x=|0_X,x|=|h(0_X),h(x)|=|0_Y,h(x)|=h(x)
$$
for any point $x\in X$. Therefore, $Y=h(X)=X$.
\end{proof}

From the above proof it follows

\begin{cor}
For $X\sse\R_+$ with $|X|\geq 2$ any isometry of the hedgehog $\hX$ is identical on the set $X$ of geometric points and gives rise to independent permutations over any geometric point.
\end{cor}

\begin{cor}\label{cor:ezhdill}
For any non-empty subset $X\sse \R_+$ the following properties are equivalent\/\rom:
\begin{enumerate}
\item Hedgehogs $\hX$ and $\l\hX$ are isometric for every $\l\in\R_+$.
\item The set $X$ coincides with $\R_+$.
\end{enumerate}
\end{cor}

\begin{proof}
According to the property~(\ref{enum:propert:2}) of Proposition~\ref{prop:propert} and Proposition~\ref{prop:ezhdist}, we have equality $X=\l X$ for every $\l \in \R_+$. The last condition is equivalent to $X=\R_+$.
\end{proof}

\begin{cor}\label{cor:isometryH}
For any non-trivial subgroup $H\sse\R_+$ the equalities
$$
\St\hH=\St_\R H=\St H=H
$$
are valid.
\end{cor}

\begin{prop}\label{prop:2ezha0dist}
Let a subset $X\sse\R_+$ and a number $M>0$ be such that $\bigl|[2M,\infty)\cap X\bigr|\ge2$. Let also discrete hedgehogs $\hX$ and $\hY$ \(for some subset $Y\sse\R_+$\/\) lie\/ \(isometrically\/\) in a metric space $Z$ so that $\hX\ss U_M(\hY)\ss Z$. Then
$$
|0_X0_Y|<4M\ \ \text{and}\ \ [5M,\infty)\cap X\ss U_M(Y)\sse\R_+.
$$
\end{prop}

\begin{proof}
Take two points $x_1,x_2\in X$ such that $x_1\ge2M$ and $x_2\ge2M$. There are points $y_1,y_2\in\hY$ such that $|x_1y_1|<M$ and $|x_2y_2|<M$. If $y_1=y_2$, then $x_1+x_2=|x_1x_2|\le|x_1y_1|+|y_2x_2|<2M$. So $y_1\ne y_2$, therefore
$$
y_1+y_2=|y_1y_2|\le|y_1x_1|+|x_1x_2|+|x_2y_2|<x_1+x_2+2M.
$$
Since $x_1\ne x_2$, the inequality $x_1+x_2<y_1+y_2+2M$ is similarly obtained.

There is a point $y_0$ such that $|0_Xy_0|<M$. It is clear that $y_i\ne y_0$, $i=1,2$, (otherwise $x_i=|x_i0_X|\le|x_iy_i|+|y_00_X|<2M$). So $y_0+y_i=|y_0y_i|\le|y_00_X|+|0_Xx_i|+|x_iy_i|<x_i+2M$. Adding the two inequalities and continuing with the previous inequality, we get
$$
2y_0+y_1+y_2<x_1+x_2+4M<y_1+y_2+6M.
$$
Therefore $y_0<3M$.

Then $|0_X0_Y|\le|0_Xy_0|+|y_00_Y|<4M$. Let now $x_1\ge5M$. If we assume that $y_1=0_Y$, then $x_1=|x_10_X|\le|x_10_Y|+|0_Y0_X|<5M$. Therefore $y_1\in Y$.
\end{proof}

\begin{cor}\label{cor:ezh0dist}
If $d_{GH}(\hX,\hY)=0$ for two subsets $X,Y\sse\R_+$ then $\bX=\bY$.
\end{cor}

\begin{proof}
According to Proposition~\ref{prop:2ezha0dist}, we have $[5r,\infty)\cap X\ss U_r(Y)$ for every $r>0$. Hence $X\ss\bY$. Similarly $Y\ss\bX$.
\end{proof}

\begin{thm}\label{thm:ezh0disthatX}
If $d_{GH}(T,\hX)=0$ for a complete metric space $T$ and a subset $X\sse\R_+$, then for some subset\/ \(with multiplicities\/\) $Y\sse\bX$ the space $T$ is isometric to the discrete hedgehog $\hY$.
\end{thm}

\begin{proof}
If a metric space is on the zero distance from a finite metric space, then the both spaces are evidently isometric. Thus, without loss of generality we assume that the set $X$ is infinite. Further, if $T$ is isometric to $\hX$, we can put $\hY=\hX$ to complete the proof. Thus, we suppose that $T$ and $\hX$ are not isometric.

Take a sequence of positive numbers $\e_n>0$ such that $\e_n\to0$. Take the space $Z$ from Corollary~\ref{cor:CollectiveIsoEmbed} for $T$, $T_n=\hX$, and $M_n=\e_n$. By the triangle inequality, we have $d_H(T_m,T_n)<\e_m+\e_n$. To apply Proposition~\ref{prop:2ezha0dist}, we assume without loss of generality that the sequence $R_n$ was chosen so that each ray $[2\e_n,\infty)$ contains at least two points of $X$ (recall that the $X$ consists of infinitely many positive real numbers). Thus, by this proposition, we have $|0_m0_n|<4(\e_m+\e_n)$, where $0_i$ denotes the center of the hedgehog $T_i$. For every $n$ we choose a point $t_n\in T$ such that $|t_n0_n|<\e_n$. Then $|t_mt_n|<5(\e_m+\e_n)$. Thus, the sequence of points $t_n$ is fundamental. It follows from the completeness of the space $T$ that this sequence has the limit $0_T\in T$.

Consider now the set of distances
$$
Y=\Bigl\{\bigl(|0_Tt|,t\bigr):t\in T,t\ne 0_T\Bigr\}\sse\R_+\x T.
$$
Let $t\in T$, $t\ne 0_T$, be an arbitrary point. Take an arbitrary number $\e>0$ such that $2\e\le|t0_T|$. Consider a set $T_n$ such that $d_H(T,T_n)<\e$ and $|0_T0_n|<\e$. There is a point $t_n\in T_n$ such that $|tt_n|<\e$. It is clear that $t_n\ne 0_n$ --- otherwise $2\e\le|t0_T|\le|t0_n|+|0_n0_T|<2\e$. Then
\begin{gather}\label{eq:4}
\left\{
\arraycolsep=1.5pt
\begin{array}{rcccl}
|t0_T|&\le&|tt_n|+|t_n0_n|+|0_n0_T|&<&t_n+2\e,\\[5pt]
t_n=|t_n0_n|&\le&|t_nt|+|t0_T|+|0_T0_n|&<&|t0_T|+2\e.
\end{array}
\right.
\end{gather}
Thus, we have shown that for any number $|0_Tt|\in Y$ and any $\e>0$, there exists a number $t_n\in X$ such that
\begin{equation}\label{eq:5}
-2\e<|0_Tt|-t_n<2\e.
\end{equation}
The latter means that $Y\sse\bX$.

The mapping $h\:T\to\hY$ arises naturally. Let us show that $h$ is an isometry. The surjectivity and injectivity of $h$ is obvious. It is also obvious that $h$ preserves distances from the point $0_T$. It remains to show that the metric in $T$ is intrinsic, i.e., for any two points $t_1,t_2\in T\setminus 0_T$ the equality $|t_1t_2|=|0_Tt_1|+|0_Tt_2|$ holds.

The inequality $|t_1t_2|\le|0_Tt_1|+|0_Tt_2|$ is an obvious consequence of the triangle inequality. The reverse inequality is obtained by simply repeating the reasoning of Proposition~\ref{prop:2ezha0dist} proof. Take an arbitrary $\e>0$ with $2\e<\min\bigl\{|0_Tt_1|,|0_Tt_2|,|t_1t_2|\bigr\}$. Fix any index $n$ such that $d_H(T,T_n)<\e$ and $|0_n0_T|<\e$. Choose points $x_i\in T_n$, $i=1,2$, for which $|t_ix_i|<\e$. Similarly to the proof of~(\ref{eq:4}), we obtain
$$
|t_i0_T|\le|t_ix_i|+|x_i0_n|+|0_n0_T|<x_i+2\e,\ \ i=1,2.
$$
Adding two inequalities, we get
$$
|t_10_T|+|t_20_T|<x_1+x_2+4\e=|x_1x_2|+4\e.
$$
According to the triangle inequality,
$$
|x_1x_2|\le|x_1t_1|+|t_1t_2|+|t_2x_2|<|t_1t_2|+2\e.
$$
So
$$
|t_10_T|+|t_20_T|-6\e<|x_1x_2|-2\e<|t_1t_2|.
$$
Since $\e$ can be chosen arbitrary small, we obtain the reverse inequality sought for, together with the required equality $|t_1t_2|=|0_Tt_1|+|0_Tt_2|$.
\end{proof}

\begin{rmk}
Proposition~\ref{prop:2ezha0dist} and Theorem~\ref{thm:ezh0disthatX} are not true for general complete spaces with one non-isolated point. The example below shows that a countable discrete complete bounded metric space can be at zero distance from a countable locally compact complete metric space with exactly one non-isolated point.
\end{rmk}

\begin{examp}\label{examp:Tuzhilin}
Consider the space $\ell_1$. Enumerate the coordinate lines (the basic vectors) by the positive integers, together with the $\infty$ symbol. The point on the $n$th coordinate line ($n=1,\ldots,\infty $) with the coordinate $x$ we denote by $(n;x)$. Consider the following $n$-point set on the $n$th coordinate line:
$$
X_n=\bigl\{(n,\,1+1/k):1\le k\le n\bigr\}.
$$
Let us form two spaces
$$
X=\bigcup_{1\le n<\infty}X_n\ \ \text{and}\ \ Y=X\cup X_\infty.
$$
\end{examp}

\begin{prop}
The $X$ and $Y$ are countable complete bounded metric spaces. The space $X$ is discrete. The space $Y$ has exactly one non-isolated point, namely, $(\infty,1)$. It holds $d_{GH}(X,Y)=0$.
\end{prop}

\begin{proof}
If we place the both $X_n$ and $X_m$ on the same coordinate line, we get $d_H(\tX_n,\tX_m)=\bigl|1/n-1/m\bigr|$. Now, for each $m<\infty$, let us construct an isometric mapping $h_m\:Y\to\ell_1$ such that $d_H\bigl(X,h_m(Y)\bigr)=1/m$. The mapping $h_m$ is defined by the formula
$$
h_m\bigl((n,x)\bigr)=
\begin{cases}
(n,x)&\text{for $n<m$};\\
(n+1,x)&\text{for $m\le n<\infty$};\\
(m,x)&\text{for $n=\infty$}.
\end{cases}
$$
The spaces $X$ and $Y$ are of the hedgehog type, i.e., they have a center and needles. In contrast to the discrete hedgehog, not only the ends of the needles are considered, but also some intermediate points on the needles. The distance in them is intrinsic, i.e., for points on different needles it is measured through the center --- the zero point. The mapping $h_m$ is an isometry on each needle, therefore it is a global isometry.
\end{proof}

%%%%%%%%%%%%%%%%%%%%%%%%%%%%%%
\section{Cloud center structure}
%%%%%%%%%%%%%%%%%%%%%%%%%%%%%%

Now we will show that the centers of different clouds can have different properties. We illustrate this difference from two sides. On the one hand, we show that centers can have different cardinalities. On the other hand, we show that the cloud center elements can have different degrees of homogeneity, i.e., their stabilizers in the isometric class can be different. The most extreme case gives the cloud $[\D_1]$ of bounded spaces: $Z[\D_1]=\{\D_1\}$, $\St\D_1=\St[\D_1]=\R_+$. The clouds $\bigl[\widehat{\Q_+}\bigr]$ and $\bigl[\widehat{\R_+}\bigr]$ give examples of other properties of the cloud center and isometric properties of spaces from the center. In~\cite{BB22T} it is shown that the clouds $\bigl[[0,\infty)\bigr]$ and $\bigl[\widehat{X_p}\bigr]$ are also interesting examples of different properties of centers. For example, among the five mentioned examples, only for the last cloud the stabilizer is not the entire group of similarities $\R_+$.

\begin{cor}\label{cor:ezh0distQ}
For a subset\/ \(with multiplicities\/\) $X\sse\R_+$ the following properties are equivalent\/\rom:
\begin{enumerate}
\item\label{enum:ezh0distQ:1} $d_{GH}\bigl(\hX,\widehat{\Q_+}\bigr)=0$.
\item\label{enum:ezh0distQ:2} The set $X$ is everywhere dense in $\R_+$, is countable and multiplicity of any point is at most $\aleph_0$.
\end{enumerate}
\end{cor}

\begin{proof}
$(\ref{enum:ezh0distQ:1})\Rightarrow(\ref{enum:ezh0distQ:2})$. According to Proposition~\ref{prop:density} and equality~(\ref{enum:propert:5}) of Proposition~\ref{prop:propert}, we have that, taking into account multiplicities of points, the set $X$ is countable.

According to Corollary~\ref{cor:ezh0dist}, we have $\bX=\overline{\Q_+}=\R_+$.

$(\ref{enum:ezh0distQ:2})\Rightarrow(\ref{enum:ezh0distQ:1})$. We fix $\e>0$. Let us partition the set of all positive numbers $\R_+$ into the countable number of subsets of diameter at most $\e$:
$$
S_n=\bigl((n-1)\,\e,\,n\e\bigr],\ \ n=1,\ldots.
$$
For every $n$, we establish a one-to-one correspondence between the countable sets $X\cap S_n$ and $\Q\cap S_n$. The constructed one-to-one correspondence of $X$ to $\Q_+$ generates a correspondence $R$ with $\dis R\le2\e$. Application of Theorem~\ref{thm:dis} completes the proof.
\end{proof}

\begin{cor}
There is equality $\St[\widehat{\Q_+}]=\R_+$.
\end{cor}

\begin{cor}
For a complete metric space $T$ the following properties are equivalent\/\rom:
\begin{enumerate}
\item\label{enum:0distQ:1} $d_{GH}\bigl(T,\widehat{\Q_+}\bigr)=0$.
\item\label{enum:0distQ:2} The space $T$ is isometric to some discrete hedgehog $\hY$ over a countable\/ \(taking into account multiplicities of points\/\) dense subset $Y\ss\R_+$.
\end{enumerate}
\end{cor}

\begin{proof}
The implication $(\ref{enum:0distQ:2})\Rightarrow(\ref{enum:0distQ:1})$ follows from Corollary~\ref{cor:ezh0distQ}.

$(\ref{enum:0distQ:1})\Rightarrow(\ref{enum:0distQ:2})$. It follows from Theorem~\ref{thm:ezh0disthatX} that $T=\hX$ for some $X\sse\R_+$. The countability and density of the set $X$ follows from the Corollary~\ref{cor:ezh0distQ}.
\end{proof}

\begin{cor}\label{cor:isometryhatQ}
If $d_{GH}\bigl(T,\widehat{\Q_+}\bigr)=0$ for a metric space $T$, then $|\St T|\le\aleph_0$.
\end{cor}

\begin{proof}
Let $T$ be complete. According to Theorem~\ref{thm:ezh0disthatX}, we can assume that $T=\hX$ for some set $X\ss \R_+$. According to Corollary~\ref{cor:ezh0distQ}, the set $X$ is countable. According to Proposition~\ref{prop:ezhdist}, the equality $\St\hX=\St X=\St_\R X$ is true. Therefore $|\St X|\le\aleph_0$.

For a non-complete space $T$, let $\tT$ denotes its completion. We can assume that $\tT=\hX$. As the space $\tT$ has just one non-isolated point, then $T=\hX\setminus 0_X=X$. Then $\St T=\St X$.
\end{proof}

\begin{thm}
\label{thm:StabilhatQ}
The center $Z[\widehat{\Q_+}]$ has the following properties\/\rom:
\begin{enumerate}
\item For every countable dense subgroup $H\sse\R_+$ the hedgehog $\hH$ is in the center $Z[\widehat{\Q_+}]$ and $\St\hH=H$.
\item $|Z[\widehat{\Q_+}]|=\fc=2^{\aleph_0}$.
\end{enumerate}
\end{thm}

\begin{proof}
The validity of the first assertion follows from Corollaries~\ref{cor:ezh0distQ} and~\ref{cor:isometryH}.

Any space from the center is a hedgehog over a countable subset of the product $(0,\infty )\x\N$. There are exactly continuum of countable subsets in this set. So the inequality $\bigl|Z[\widehat{\Q_+}]\bigr|\le\fc=2^{\aleph_0}$ is proved. The group $\R_+$ has exactly continuum of dense countable subgroups, so the reverse inequality follows from the first assertion.
\end{proof}

\begin{rmk}
In fact, for any subgroup $|H|\le\aleph_0$ there exists a space $T_H=\widehat{X_H}\in Z[\widehat{\Q_+}]$ such that $\St\widehat{X_H}=H$. Since for a non-dense subgroup the countable dense subset $X_H\ss\R_+$ is constructed by a more complicated procedure, then we drop it.

Corollary~\ref{cor:ezh0distQ} can be transformed into a criterion for the distance between two hedgehogs to be equal to zero. We have such a criterion is needed only in the case of the maximum hedgehog, where we decided to give a direct proof for a special case.
\end{rmk}

\begin{cor}\label{cor:0distR}
For a subgroup $H\ss \R_+$ the following properties are equivalent\/\rom:
\begin{enumerate}
\item\label{enum:0distR:1} $d_{GH}\bigl(\hH,\widehat{\R_+}\bigr)=0$.
\item\label{enum:0distR:2} $|H|=\fc=2^{\aleph_0}$.
\end{enumerate}
\end{cor}

\begin{proof}
$(\ref{enum:0distR:1})\Rightarrow(\ref{enum:0distR:2})$. According to  Proposition~\ref{prop:density} and equality~(\ref{enum:propert:5}) of Proposition~\ref{prop:propert}, we have $|H|=d(\hH)=d(\widehat{\R_+})=|\R_+|=\fc$.

$(\ref{enum:0distR:2})\Rightarrow(\ref{enum:0distR:1})$.
Every subgroup of cardinality $\fc$ is uniformly dense, that is $|H\cap U|=\fc$ for every open set $U\ss\R_+$. The proof is now completed by repeating the argument from the proof of Corollary~\ref{cor:ezh0distQ}.
\end{proof}

\begin{thm}\label{thm:StabilhatR}
The center $Z[\widehat{\R_+}]$ has the following properties\/\rom:
\begin{enumerate}
\item For every subgroup $H\sse\R_+$ of cardinality $\fc$, the hedgehog $\hH$ is in the center $Z[\widehat{\R_+}]$ and $\St\hH=H$.
\item $\bigl|Z[\widehat{\R_+}]\bigr|=2^\fc=2^{2^{\aleph_0}}$.
\end{enumerate}
\end{thm}

\begin{proof}
The validity of the first assertion follows from Corollaries~\ref{cor:0distR} and~\ref{cor:isometryH}.

Any space from the center is a hedgehog over a subset of the product $(0,\infty)\x(0,\infty)$. There are exactly $2^\fc$ subsets in this set. So the inequality $\bigl|Z[\widehat{\Q_+}]\bigr|\le2^\fc$ is proved. The group $\R_+$ has exactly $2^\fc$ of subgroups of cardinality $\fc$, so the reverse inequality follows from the first assertion.
\end{proof}

\begin{rmk}
The center $Z[\widehat{\R_+}]$ contains the hedgehog itself, so it contains space with the maximum possible isometric stabilizer. The center $Z[\widehat{\Q_+}]$ does not contain space with the maximum possible isometric stabilizer.
\end{rmk}

\end{document}